\documentclass[reqno, 12pt]{amsart}
\usepackage[letterpaper,hmargin=1.5in,vmargin=3cm]{geometry}
\usepackage{amsmath,amssymb,amsthm,}
\usepackage{verbatim,epsfig,graphics,hyperref}
\usepackage{marginnote}
\usepackage{color}
\newtheorem{theorem}{Theorem}[section]
\newtheorem{proposition}[theorem]{Proposition}
\newtheorem{corollary}[theorem]{Corollary}

\newtheorem{example}{Example}

\newtheorem{remark}[theorem]{Remark}
\newtheorem{conjecture}[theorem]{Conjecture}

\begin{document}
\title[Nodal count for Dirichlet-to-Neumann operators]
{Nodal count for Dirichlet-to-Neumann operators with potential
}
\author{Asma Hassannezhad}
\address{University of Bristol,
School of Mathematics,
Fry Building,
Woodland Road,
Bristol, 
BS8 1UG, U.K.}
\email{asma.hassannezhad@bristol.ac.uk}
\author{David Sher}
\address{DePaul University, Department of Mathematical Sciences, 2320 N Kenmore Ave., Chicago, IL, 60614, U.S.A.}
\email{dsher@depaul.edu}
\date{}
\begin{abstract} We consider Dirichlet-to-Neumann operators associated to $\Delta+q$ on a Lipschitz domain in a smooth manifold, where $q$ is an $L^{\infty}$ potential. We prove a Courant-type bound for the nodal count of the extensions $u_k$ of the $k$th Dirichlet-to-Neumann eigenfunctions $\phi_k$ to the interior satisfying $(\Delta+q)u_k=0$.  The classical Courant nodal domain theorem is known to hold for Steklov eigenfunctions, which are the harmonic extension of the Dirichlet-to-Neumann eigenfunctions associated to $\Delta$. Our result extends it to a larger family of Dirichlet-to-Neumann operators.  Our proof makes use of the duality between the Steklov and Robin problems.
\end{abstract}
\maketitle
\noindent \\
\section{Introduction}
We consider Dirichlet-to-Neumann operators associated to the Laplace operator with a potential. Let $M$ be a smooth Riemannian manifold, $\Omega\subseteq M$ a connected Lipschitz domain, and $q\in L^{\infty}(\Omega)$ a potential function. Consider the operator $\Delta_q := \Delta +q$ on $\Omega$, where $\Delta=-{\rm div}\, \nabla$ is the positive Laplacian. Denote by $\Delta_q^D$ the operator $\Delta_q$ with Dirichlet boundary condition on $\partial\Omega$. The operator $\Delta_q^D$ has discrete spectrum whose only accumulation point is $+\infty$.

Now let $\lambda\in\mathbb R$. 
We consider the Dirichlet-to-Neumann operator $\mathcal D_{q,\lambda}$ associated to $\Delta_q-\lambda$.  We first define this in the case where $\lambda$ is not an eigenvalue of $\Delta_q^D$. In that case, for any $g\in L^2(\partial\Omega)$, the equation
\[\begin{cases} \Delta_{q,\lambda}u = 0 & \textrm{ in }\Omega\\
 u = f & \textrm{ on }\partial\Omega\\
\end{cases}
\]
has a unique solution $u$, and we set
\[\mathcal D_{q,\lambda}f := \partial_nu,\]
where $\partial_nu$ is the outward pointing normal derivative of $u$ along $\partial\Omega$.
If $\lambda$ is an eigenvalue of $\Delta_q^D$, the solution is no longer unique, but we may still define $\mathcal D_{q,\lambda}$ by projecting off the subspace consisting of normal derivatives of Dirichlet eigenfunctions. 
As we will see, in either event, $\mathcal D_{q,\lambda}$ is a semi-bounded self-adjoint operator and  has discrete, real spectrum whose only accumulation point is $+\infty$. We denote its eigenvalues, with multiplicity, by $\{\sigma_k\}_{k=1}^{\infty}$, and fix a corresponding basis of eigenfunctions for $L^2(\partial\Omega)$ by $\{\phi_k\}_{k=1}^{\infty}$. Finally, define $\{u_k\}_{k=1}^{\infty}$ to be the interior extensions of $\phi_k$, that is, the functions for which
\[\begin{cases} (\Delta_{q}-\lambda)u_k = 0 & \textrm{ in }\Omega\\
 u_k = \phi_k & \textrm{ on }\partial\Omega,\\
\end{cases}
\]
again with the appropriate modifications when $\lambda$ is an eigenvalue of $\Delta_q^D$. As in the case $q=0$, we call $\{u_k\}$ the corresponding Steklov eigenfunctions.

In this paper, we discuss the nodal counts of both the Steklov eigenfunctions $u_k$ and the Dirichlet-to-Neumann eigenfunctions $\phi_k$. Throughout, we let $N_k$ be the number of nodal domains of $u_k$ on $\Omega$ and let $M_k$ be the number of nodal domains of $\phi_k$ on $\partial\Omega$.

In the case $q=0$, it is well-known that we have an analogue of the Courant nodal domain theorem for Steklov eigenfunctions (see~\cite{KS69,KKP14,GP17}). Specifically, \begin{equation*}\label{courntstek}N_k\leq k.\end{equation*} In this case, the
proof essentially uses three ingredients: the variational principle for eigenvalues, the unique continuation theorem for the solutions of a second order elliptic PDE, and the fact that harmonic functions are the unique minimizers of the Dirichlet energy for given boundary data.   The statement  does not hold when $q$ is an arbitrary nonzero potential. However, as we show, there is a replacement:
\begin{theorem}\label{thm:main1} With terminology as above, let $d$ be the number of non-positive Dirichlet eigenvalues of $\Delta_{q,\lambda}$, or equivalently the number of eigenvalues of $\Delta_q^D$ which are less than or equal to $\lambda$. Then for all $k\in\mathbb N$,
\[N_k\leq k + d.\]
\end{theorem}
\begin{remark}\label{remarksharp} This theorem is sharp in the sense that for any $d\in\mathbb N$, there exists a domain $\Omega$, a potential function $q$, and an integer $k$ for which $N_k=k+d$.
\end{remark}
\begin{remark}
     If $\Omega$ is a fixed subdomain of $\mathbb R^n$ and $q$ is sufficiently small, then perturbation theory (see e.g. \cite[Page 76]{Rellich}) implies that $\Delta_q$ has only positive Dirichlet eigenvalues. The same is true when $q\ge0$. Thus, by Theorem \ref{thm:main1}, $N_k\leq k$ for the operator $\mathcal D_{q,0}$ in these cases.
\end{remark}

Very little is known about  the nodal count of the Dirichlet-to-Neumann eigenfunctions $\phi_k$. See Open Problem 9 in \cite{GP17}. The statement that $M_k\leq k$ is certainly not true in general, for the same reasons as for $N_k$. In fact, the situation is worse, as $\partial\Omega$ may be disconnected, in which case, even if $q=0$, the Courant nodal domain theorem cannot hold for the ground state $k=1$. When $q=0$ and the dimension of $\Omega$ is two, the fact that no nodal line is a closed curve implies an estimate on $M_k$ in terms of $k$ and the topology of the domain. For example, for a simply connected domain, the bound is $2k$ \cite[Lemma 3.4]{AM94}. However, for $q\neq0$ no such bound exists. See Example \ref{keyexample} below. In higher dimension, nothing is known regarding bounds for $M_k$ even when $q=0$. The main difficulty is that $\mathcal{D}_{q,\lambda}$  is nonlocal and the method of the proof we employ to study the nodal count of $u_k$ cannot be generalised to  study the nodal count of the Dirichlet-to-Neumann eigenfunctions $\phi_k$.

We conjecture the following \emph{asymptotic} version of the Courant nodal domain theorem:
\begin{conjecture}\label{conj:wishlist}
With terminology as above,
\[\limsup_{k\to\infty}\frac{M_k}{k}\leq 1.\]
\end{conjecture}

\begin{remark}
     Note that the corresponding result for $N_k$,
\[\limsup_{k\to\infty}\frac{N_k}{k}\leq 1,\]
follows immediately from Theorem \ref{thm:main1}.
\end{remark}

\begin{remark} If Conjecture \ref{conj:wishlist} is true, it would immediately imply 
\begin{equation}\label{oq9gp}
    M_k\le k+o(k).
\end{equation}
This would yield a partial answer to Open Question 9 in \cite{GP17}.
\end{remark}

We also conjecture the following sharpened version in dimension at least three. This is motivated by the Pleijel theorem for the nodal count of the Laplace operator \cite{BM82, Pl56}. 
\begin{conjecture}\label{conj:wishlist2}
When the dimension of $\Omega$ is at least three, \[\limsup_{k\to\infty}\frac{M_k}{k}< 1\]
and
 \[\limsup_{k\to\infty}\frac{N_k}{k}< 1.\]
\end{conjecture}

In fact, this sharpened version is true in a number of special cases. For example, suppose that $\Omega$ is a cylinder $[0,1]\times \Sigma$, where $\Sigma$ is a compact manifold of dimension at least two. One can use separation of variables and Pleijel's theorem \cite{BM82, Pl56} to show that
\[\limsup\frac{M_k}{k}\leq c<1.\]
The same result is true if $M_k$ is replaced by $N_k$. A similar result holds if $\Omega$ is a ball in $\mathbb R^n$, with $n\geq 3$.\\

The key example to keep in mind is the following, motivated by \cite[Figure 1]{GKLP21}. In particular, it shows that $\frac{N_k}{k}$ and $\frac{M_k}{k}$ are only \emph{asymptotically}  bounded by one.
\begin{example}\label{keyexample}
Let $\Omega$ be the unit disk, set $\lambda=0$, and let $q$ be the constant function $-\mu$ for some $\mu\ge 0$. Then the spectrum of $\mathcal D_{q,\lambda}$ is of the form
\begin{equation}
    \left\{\frac{\sqrt{\mu}J_n'(\mu)}{J_n(\mu)},\ n\in\mathbb N_0\right\},
\end{equation}
with a corresponding basis of eigenfunctions $J_n(\sigma r)e^{\pm in\theta}$ \cite{GKLP21}. 

Note that $J_n(\mu)$ is zero if and only if $\mu$ is a Dirichlet eigenvalue of $\Delta+q=\Delta-\mu$. So fix a particular $n$ and consider what happens as $\mu$ approaches the first zero $j_{n,1}$ of $J_n(x)$ from below. The eigenvalue of $\mathcal D_{-\mu,\lambda}$ corresponding to that particular $n$ will go to $-\infty$. (It is simple if $n=0$ and double if $n>0$.) Since the Dirichlet eigenvalues of a disk all have multiplicity at most 2, all other eigenvalues stay bounded below. If we choose
\[\mu = j_{n,1}-\epsilon\]
for a sufficiently small $\epsilon>0$, then the smallest eigenvalue of $\mathcal D_{-\mu,\lambda}$ will be $\sigma=\frac{\sqrt{\mu}J_n'(\mu)}{J_n(\mu)}$, with eigenfunction(s) $J_n(\sigma r)e^{\pm in\theta}$. So these eigenfunction(s) are the ground state eigenfunction(s) for $\mathcal D_{-\mu,\lambda}$, i.e. they have $k=1$. However, each of them has $n$ boundary nodal domains and $n$ interior nodal domains as well, so we have $N_k=M_k=n$. Since $n$ is arbitrary, not only can we have $N_k>k$, but we can have as large a discrepancy as we like, illustrating the sharpness in Remark \ref{remarksharp}.
\end{example}

The key method for the proof of Theorem \ref{thm:main1} is to make use of Steklov-Robin duality. This is the observation that the two-parameter problem
\[\begin{cases} \Delta_{q}u = \lambda u & \textrm{ in }\Omega\\
\partial_n u = \sigma u & \textrm{ on }\partial\Omega\\
\end{cases}
\]
may be viewed either as a Steklov problem for fixed $\lambda$, with eigenvalue parameter $\sigma$, or as a Robin problem for fixed $\sigma$, with eigenvalue parameter $\lambda$. This idea has a long history, at least in the case $q=0$. It was first written down in \cite{GNP} but seems to have been known to others, including Caseau and Yau (see the discussion in \cite{AM2}). In 1991, L.~Friedlander rediscovered it and used it to give a proof of the interlacing of Dirichlet and Neumann eigenvalues for domains in $\mathbb R^n$ \cite{Fri, Maz91}. In \cite{AM1,AM2}, Arendt and Mazzeo generalized the Steklov-Robin duality to manifolds; though Friedlander's inequalities fail in that setting, the duality results themselves still hold. Some duality results with nonzero potential, though nominally in the Euclidean setting only, are given in \cite{AEKS14}. Finally, we should note that Steklov-Robin duality has been used to compare Steklov eigenvalues and eigenvalues of the boundary Laplacian, see for example \cite{GKLP21}, which gave us the idea for Example \ref{keyexample}.

\section{Modified Courant nodal domain theorem for  Steklov eigenfunctions with potential}\label{sect2}

Let $\Omega$ be a Lipschitz domain in a smooth Riemannian manifold $M$. Let $q\in L^{\infty}(\Omega)$ be a potential. It is enough to prove Theorem \ref{thm:main1} for $\lambda=0$, as $\lambda$ may be absorbed into the potential $q$. Therefore, Theorem \ref{thm:main1} is an immediate consequence of the following result:

\begin{theorem}\label{courantplusd} Suppose that $\Omega$ and $q$ are as above. 
Suppose that $\{u_k\}_{k=1}^\infty$ is a complete set of Steklov eigenfunctions for $\Delta_q$, and $N_k$ is the number of nodal domains of $u_k$ on $\Omega$. Then
\[N_k \leq k + d,\]
where $d$ is the number of non-positive eigenvalues of the following Dirichlet eigenvalue problem:
\[\begin{cases}  \Delta_qu=\lambda u&\textrm{in}~\Omega\\
u=0& \textrm{on}~\partial \Omega.\end{cases}\]
\end{theorem}

The proof of Theorem \ref{courantplusd} uses Steklov-Robin duality. 
Let introduce two parameters, $\lambda$ and $\sigma$, and consider the problem
\[\begin{cases}\Delta_qu=\lambda u,&\text{in}~\Omega\\
 \partial_nu=\sigma u,& \text{on}~\partial \Omega.\end{cases}\]
One may consider $\lambda$ as the spectral parameter, in which case we have a Robin problem with fixed $\sigma$, or consider $\sigma$ as the spectral parameter, in which case we have a Steklov-type problem with fixed $\lambda$.  We let $\lambda_{q,k}(\sigma)$ be the $k$th eigenvalue of $\Delta_{q,\sigma}$. Observe that the $k$th Steklov eigenfunction $u_k$ is an eigenfunction with eigenvalue $\lambda=0$ for the Robin problem:
\[\begin{cases}\Delta_qu=\lambda u&\text{in}~\Omega\\
 \partial_nu=\sigma_ku& \text{on}~\partial \Omega.\end{cases}\]
The question is for which $m$ $\lambda_{q,m}(\sigma_k)$ is equal to 0.

The duality results we need essentially follow from \cite{AM1}, \cite{AM2}, and \cite{AEKS14}. However, they are not stated in quite this much generality, and so we give a proof here. Our approach is modeled primarily on \cite{AM2}.

First, we define the Robin Laplacian $\Delta_{q,\sigma}$ by using the weak formulation. Consider the form, for $u$, $v\in H^1(\Omega)$,
\[b_{q,\sigma}(u,v) = \int_{\Omega}(\nabla u\cdot\overline{\nabla v} + qu\overline{v})\, dV_{\Omega} - \sigma\int_{\partial\Omega}u\overline{v}\, dV_{\partial\Omega}.\]
Since $q\in L^{\infty}$, this form is coercive, and so it determines an operator $\Delta_{q,\sigma}$, which is the Robin Laplacian. The domain of $\Delta_{q,\sigma}$ is the same as the domain of $\Delta_{0,\sigma}$, namely
\[\{u:u\in L^2(\Omega), \Delta u\in L^2(\Omega), \partial_n u = \sigma u ~\text{on $\partial\Omega$}\}.\]
A Dirichlet Laplacian with potential, $\Delta_{q}^D$, may also be defined as usual.

For each $\lambda$ which is not in the spectrum of $\Delta_q^D$, we define the Dirichlet-to-Neumann operator $\mathcal D_{q,\lambda}$. If $g\in L^2(\partial\Omega)$ and $u\in H^1(\Omega)$ 
is the unique solution of 
\[\begin{cases} (\Delta_{q}-\lambda)u =0 & \textrm{ in }\Omega\\
 u = f& \textrm{ on }\partial\Omega\\
\end{cases}
\]
then we set $\mathcal D_{q,\lambda}f=\partial_n u$. This is enough for many purposes. However, we need to consider $\lambda$ which are in the Dirichlet spectrum of $\Delta_q$. There are several ways to do this, the simplest of which is to restrict to the orthogonal complement of the kernel. Following \cite{AM2}, we define
\[K(\lambda) = \{\partial_nw:\Delta_{q}w =  \lambda w~\text{weakly}, w|_{\partial\Omega}=0, \partial_nw\in L^2(\partial\Omega)\}.\]
Then $\mathcal D_{q,\lambda}$ may in all cases be defined as an operator on $L_{\lambda}^2(\partial\Omega):=(K(\lambda))^{\perp}$, \label{L_lambda} where the orthogonal complement is taken in $L^2(\partial\Omega)$.

\begin{proposition} For any $\lambda\in\mathbb R$, $\mathcal D_{q,\lambda}$ is self-adjoint, has compact resolvent, and is bounded below.
\end{proposition}
\begin{proof} A proof is given in \cite{AM2} when $q=0$. However, it depends on a result of Gr\'egoire, N\'ed\'elec, and Planchard \cite{GNP}, which is only stated in the setting $q=0$. We instead use the machinery of \cite{AEKS14}, which instead views the Dirichlet-to-Neumann operator as a graph, that is, as a multi-valued operator. From \cite[Proposition 3.3]{AEKS14}, it suffices to prove that this graph is self-adjoint, has compact resolvent, and is bounded below. Yet this is essentially the content of \cite[Example 4.9]{AEKS14}. Although stated in the setting $M=\mathbb R^n$ and $\lambda=0$, every assertion there holds when $M$ is an arbitrary Riemannian manifold, and a nonzero $\lambda$ may be treated as part of the potential. The three parts of our Proposition then follow from Theorem 4.5, Proposition 4.8, and Theorem 4.15 of \cite{AEKS14}.
\end{proof}

As a consequence, the spectrum of $\mathcal D_{q,\lambda}$ is contained in the real axis, discrete, and has only the accumulation point at infinity.\\

In what follows we use the notational conventions:
\[\mathcal{D}_{q,0}:=\mathcal{D}_q,\quad \mathcal{D}_{0,0}=:\mathcal{D}.\]
Obviously, we have $\mathcal{D}_{q,\lambda}=\mathcal{D}_{q-\lambda}$. However, it will be convenient to separate the role of  $\lambda$ from the potential $q$ to  highlight the connection between $D_{q,\lambda}$ and the Robin problem.

The following proposition, encapsulating the Steklov-Robin duality, is the analogue of \cite[Theorem 3.1]{AM2} and is proved in identical fashion.
\begin{proposition}\label{kernels}
For any $\lambda,\sigma\in\mathbb R$, the trace map is an isomorphism from $\ker(\Delta_{q,\sigma}-\lambda)$ to $\ker(\mathcal D_{q,\lambda}-\sigma)$.
\end{proposition}
\begin{remark}
Since $\mathcal D_{q,\lambda}=\mathcal D_{q-\lambda}$,   Proposition  \ref{kernels} is equivalent to show that the trace map is an isomorphism from $\ker(\Delta_{q,\sigma})$ to $\ker(\mathcal D_{q}-\sigma)$ for any $\sigma\in\mathbb{R}$ and $q\in L^\infty(\Omega)$.
\end{remark}
\begin{proof}
First we show that the trace map indeed maps into the indicated space. Suppose that $u\in\ker(\Delta_{q,\sigma}-\lambda)$. Then $u\in H^1(\Omega)$, so certainly Tr$(u)\in L^2(\partial\Omega)$. Since $u$ is in the domain of $\Delta_{q,\lambda}$, $\partial_nu$ exists and equals $\sigma$Tr$(u)$. And as long as Tr$(u)\in (K(\lambda))^{\perp}$, it is in the domain of $\mathcal D_{q,\lambda}$. In that event, we can say that $\mathcal D_{q,\lambda}($Tr$(u))=\sigma$Tr$(u)$, hence Tr$(u)\in\ker(\mathcal D_{q,\lambda}-\sigma)$.

To show that Tr$(u)\in(K(\lambda))^{\perp}$, suppose that $\partial_nw\in K(\lambda)$, with $w\in\ker(\Delta_{q}^D-\lambda)$. Then by Green's identity,
\[\langle\partial_nw,\textrm{Tr}(u)\rangle_{L^2(\partial\Omega)} = \langle\nabla w,\nabla u\rangle_{L^2(\Omega)} - \langle\Delta_{0,D}w, u\rangle_{L^2(\Omega)}. \]
Using Green's identity again, combined with the facts that $w$ is in the domain of the Dirichlet Laplacian and $w\in\ker(\Delta_{q}^D-\lambda)$, we have
\begin{equation}\label{Green2}
\langle\partial_nw,\textrm{Tr}(u)\rangle_{L^2(\partial\Omega)} = \langle w,\Delta u\rangle_{L^2(\Omega)} - \langle(\lambda-q)w, u\rangle_{L^2(\Omega)}.
\end{equation}
However, since $u\in\ker(\Delta_{q,\sigma}-\lambda)$, we know that $\Delta u = (\lambda-q)u$. Since $\lambda\in\mathbb R$ and $q$ is real-valued, the right-hand side of \eqref{Green2} is zero. Thus Tr $(u)\in (K(\lambda))^{\perp}$ and therefore the trace map does indeed map into $\ker(\mathcal D_{q,\lambda}-\sigma)$.

To show that the trace map is injective, suppose that $u\in\ker(\Delta_{q,\sigma}-\lambda)$ with Tr$(u)=0$. From our definition of $\mathcal D_{q,\lambda}$, we know that for any $g\in (K(\lambda))^{\perp}$, the problem
\[\begin{cases} (\Delta_{q}-\lambda)u =0 & \textrm{ in }\Omega\\
 u = g& \textrm{ on }\partial\Omega\\
\end{cases}
\]
has a unique solution whose trace is in $(K(\lambda))^{\perp}$. 
Since both $u$ and $0$ are solutions to this problem with $g=0$, we must have $u=0$.

Finally, surjectivity is straightforward: suppose that $g\in\ker (\mathcal D_{q,\sigma}-\lambda)$. By definition there is a function $u\in H^1(\Omega)$ such that $(\Delta_q-\lambda)u=0$, Tr$(u)=g$, and $\partial_n u =\sigma g$. This $u$ is an element of $\ker(\Delta_{q,\sigma}-\lambda)$ whose trace is $g$.
\end{proof}
An immediate consequence is
\begin{corollary}\label{AM:DtNRobin}
For any $\lambda,\sigma\in\mathbb R$, $\sigma$ is an element of the (Steklov) spectrum of $\mathcal D_{q,\lambda}$ if and only if $\lambda$ is an element of the (Robin) spectrum of $\Delta_{q,\sigma}$. Moreover their geometric multiplicities are the same.
\end{corollary}

The following statement describes its behaviour as $\sigma$ varies.
\begin{proposition}\label{lambda_q,k} For every $k\ge1$ the following hold:
\begin{itemize}
\item[(a)] $\lambda_{q,k}(\sigma)$ is strictly decreasing. 
\item[(b)] $\lambda_{q,k}$ as a function of $\sigma$ is continuous on $[-\infty,\infty)$. In particular,
\[\lim_{\sigma\to-\infty}\lambda_{q,k}(\sigma)=\lambda_{q,k}^D,\]
where $\lambda_k^D$ is the $k$-th Dirichlet eigenvalue of $\Delta_{q}^D$.
\item[(c)] $\lim_{\sigma\to\infty}\lambda_{q,k}(\sigma)=-\infty.$
\end{itemize}
\end{proposition}
The proof of this proposition follows, nearly verbatim, the proof presented in \cite[Proposition 3]{AM2} and \cite[Section 2]{AM1}. For the sake of completeness and the reader's convenience, we give the proof. 
\begin{proof}

To prove a), note that by the max-min principle for eigenvalues, we have $$\lambda_{q,k}(\sigma)=\sup_{V_{n-1}}\inf\{b_{q,\sigma}(u): u\in V_{n-1}, \|u\|_{L^2(\Omega)}=1\}$$
where the supremum is taken over all subspaces $V_{n-1}\subset H^1(\Omega)$ of codimension $n-1$.  Since $b_{q,\sigma}(u)$ is strictly decreasing in $\sigma$, it follows that $\lambda_{q,k}(\sigma)$ is decreasing. To show that it is strictly decreasing, assume to the contrary that for some $\sigma<\tilde \sigma$,  $\lambda_{q,k}(\sigma)=\lambda_{q,k}(\tilde\sigma)$. It implies that $\lambda:=\lambda_{q,k}(\sigma)$ is constant on $[\sigma,\tilde\sigma]$. By Corollary \ref{AM:DtNRobin}, $[\sigma,\tilde\sigma]$ must be a subset of the spectrum of $\mathcal{D}_{q,\lambda}$. This contradicts the fact that the spectrum of $\mathcal{D}_{q,\lambda}$ is discrete. 

To prove b), we first show the continuity of the resolvents  $(\mu+\Delta_{q,\sigma})^{-1}$, $\sigma\in[-\infty,\infty)$. It was shown for sufficiently large $\mu$, in the $q=0$ case, in \cite[Proposition 2.6]{AM1}. The proof remains the same and the statement remains true. Thus, for $\mu$ large enough, 
\[\lim_{s\to\sigma}(\mu+\Delta_{q,s})^{-1}=(\mu+\Delta_{q,\sigma})^{-1}\]
and in particular when $\sigma=-\infty$, $\Delta_{q,-\infty}=\Delta_{q}^D$. Hence \[\lim_{s\to-\infty}(\mu+\Delta_{q,s})^{-1}=(\mu+\Delta_{q}^D)^{-1}.\]
We can now use \cite[Proposition 2.8]{AM1} to conclude that  for every $k\ge 1$ and $s\in[-\infty,\infty)$,
\[\lim_{\sigma\to s}\lambda_{q,k}(\sigma)=\lambda_{q,k}(s).\]
In particular, for $s=-\infty$
\[\lim_{\sigma\to-\infty}\lambda_{q,k}(\sigma)=\lambda_{q,k}^D.\]

For c), assume that $\lambda_{q,k}(\sigma)$ is bounded below by some $\lambda\in \mathbb{R}$ for all $\sigma\in \mathbb{R}$, i.e. $$\lambda_{q,k}(\sigma)>\lambda,\qquad \sigma\in \mathbb{R}.$$
Note that $\lambda< \lambda_{q,k}(\sigma)\le\lambda_{q,k+1}(\sigma)$ for all $\sigma$. 
By Corollary \ref{AM:DtNRobin} we have that the spectrum of $\mathcal{D}_{q,\lambda}$ is the set  $$\{\sigma\in \mathbb{R}: \lambda=\lambda_{q,j}(\sigma)~~\text{for some $j=1,\cdots,k-1$}\}.$$ However, this set is finite by part a). This is impossible. \end{proof}

\begin{proposition}\label{d+k}
For any $\lambda\in \mathbb{R}$, consider $d\in\mathbb{N}\cup\{0\}$ such that $\lambda_{q,d}^D\le \lambda< \lambda_{q,d+1}^D$. By  convention $\lambda_{q,0}^D=-\infty$. Then 
for every $k\ge 1$, there exists  a unique $s_k\in \mathbb{R}$ such that $\lambda_{q,k+d}(s_k)=\lambda$. Moreover, $s_k=\sigma_k(\mathcal{D}_{q,\lambda})$ for every $k\ge1$.
\end{proposition}
The proof follows the same line of argument as in the proof of Proposition 4.5 in \cite{AM1}; see also \cite[Proposition 4]{AM2}.
\begin{proof}
By Proposition \ref{lambda_q,k}, we have $$\lim_{\sigma\to-\infty}\lambda_{q,k+d}(\sigma)=\lambda_{q,k+d}^D>\lambda\ge \lambda_{q,d}^D,\qquad \lim_{\sigma\to\infty}\lambda_{q,k}(\sigma)=-\infty,$$ for every $k\ge1$. Thus, the existence and uniqueness of $s\in\mathbb{R}$ follows from the fact that $\lambda_{q,k+d}$ is a strictly decreasing continuous function.
If $s\in \{\sigma_j(\mathcal{D}_{q,\lambda})\}$, then there exists $m\in \mathbb{N}$ such that $\lambda_{q,m}(s)=\lambda$. Hence, $m\ge d+1$
 and $s=s_{k}$, where $k=m-d$.  Indeed, thanks to Proposition \ref{lambda_q,k}, for every $m\le d$ and $s\in\mathbb{R}$, $\lambda_{q,m}(s)<\lambda^D_{q,d}\le\lambda.$  This shows that $\{\sigma_j(\mathcal{D}_{q,\lambda})\}$ and $\{s_j\}$ are equal as sets. It remains to show that they are equal as multisets, i.e. their multiplicities are equal. 
 
 It is easy to observe that $s_k\le s_{k+1}$. Indeed, if $s_k>s_{k+1}$, then $$\lambda_{q,d+k+1}(s_{k+1})=\lambda=\lambda_{q,d+k}(s_{k})<\lambda_{q,d+k}(s_{k+1})\le \lambda_{q,d+k+1}(s_{k+1})$$ gives a contradiction. Assume that $s_k$ has multiplicity $p$ and $s:=s_k<s_{k+p}$. Hence, $\lambda_{q,k+d+j}(s)=\lambda$, $j=0,\ldots,p-1$.  But $\lambda_{q,k+d+p}(s)<\lambda_{q,k+d+p}(s_{k+p})=\lambda$. Thus the multiplicity of $\lambda_{q,k+d}(s)$ is at least equal to $p$. If $k=1$, then $\lambda_{q,d-1}(s)<\lambda^D_{q,d-1}\le\lambda$. If $k>1$, by assumption $s_{k-1}<s_k$ and $\lambda_{q,d+k-1}(s)=\lambda_{q,d+k-1}(s_{k})<\lambda=\lambda_{q,d+k-1}(s_{k-1})$. Therefore, in both cases, the multiplicity of $\lambda_{q,k+d}(s)$ is equal to $p$ and so,  by Proposition~\ref{kernels}, is the multiplicity of $\sigma_k(\mathcal{D}_{q,\lambda})$. 
 \end{proof}

Theorem \ref{courantplusd} is now an immediate consequence of Propositions  \ref{d+k} and \ref{kernels}.

 \section*{Acknowledgements}
The authors are grateful to Graham Cox and Alexandre Girouard for helpful discussions.
A.H. gratefully acknowledges the support from  EPSRC grant EP/T030577/1. D.S. is grateful for the support of an FSRG grant from DePaul University.

\bibliographystyle{plain}
\bibliography{ref}
\end{document}